\documentclass[12pt]{amsart}
\newtheorem{lemma}{Lemma}[section]
\newtheorem{theorem}[lemma]{Theorem}

\newtheorem{exercise}[lemma]{Exercise}

\usepackage{latexsym}
\begin{document}
\title{Mathematical Musings of a Urologist}
\author{John R. Akeroyd; Robert K. Powers; Ganesh Rao, M.D.}
\date{August 1, 2023\\}

\begin{abstract} The derivatives with respect to the variable $r$ of $\pi r^2$ and $\frac{4}{3}\pi r^3$ are
$2\pi r$ and $4\pi r^2$, respectively. This relates, through the derivative, the area enclosed in a circle to the 
length of that circle and, likewise, the volume of a sphere to the surface area of that sphere. The reasons why 
this works are basic to a first course in calculus. In this brief article, we expand on these ideas to shapes other 
than circles and spheres. Our approach is with the first year calculus student in mind.
\end{abstract}
\maketitle

\section{Introduction and the Reasons Why}
\bigskip
Urologists are natuarally interested in surface area versus rate of change in volume of spherically shaped objects. 
We review connections between these things without pretending that anything we say in this brief article is fundamentally new.
In fact, numerous articles have been written on this topic -- we mention two: \cite{JM} and \cite{JT}.
Our main result, Theorem 1.4, is accessible to the first year calculus student. As mentioned in the abstract, the circumference of a circle of radius $r$
is the derivative with respect to $r$ of the formula for the area enclosed in that circle, and likewise with the surface area of 
a sphere of radius $r$ and the formula for the 
volume of the region bounded by that sphere. In the case of the circle of radius $r$ (see \cite{WT}),
$$\pi r^2 = \int_0^{2\pi} \int_0^r \rho\, d\rho\,d\theta = \int_0^r 2\pi \rho\, d\rho.$$
And in the case of the sphere,
$$\frac{4}{3}\pi r^3 = \int_0^\pi\int_0^{2\pi}\int_0^r \rho^2\sin{\varphi}\,d\rho\, d\theta\, d\varphi   = \int_0^r 4\pi \rho^2\, d\rho.$$
In the first case, the integral in terms of polar coordinates gives the area of the  region enclosed by a circle of radius $r$ as an integral of the 
circumference formula for a circle of radius $\rho$, $0 < \rho < r$. And in the second, the integral in terms of spherical coordinates gives
the volume of the region bounded by a sphere of radius $r$ as an integral of the surface area formula for a sphere of radius $\rho$, $0 < \rho < r$. 
So, it is as basic as the Fundamental Theorem of Calculus, and certainly no accident, that the derivatives of the area and volume formulae here are 
circumference and surface area formulae, respectively. In the case of a square of sidelength $x$, the area enclosed is $x^2$ and the perimeter is $4x$. 
Clearly, the derivative of $x^2$ does not equal $4x$. Things work out nicely, though, if we express terms differently: In a way that is more akin to polar 
and spherical coordinates. Given a square $S$, let $r$ be the radius of the largest circle that can be inscribed in $S$. Then the sidelength of $S$ is $2r$ 
and the area of the region enclosed by $S$ is $4r^2$. So, the derivative with respect to $r$ of the area  of the region enclosed by $S$ is $8r$, which is 
the length of $S$ as a curve. And if $T$ is an equilateral triangle and $r$ is the radius of the largest circle that can be inscribed in $T$, then one-half the 
base of $T$ is $\sqrt{3}r$ and the height of $T$ is $3r$; and so the area of the region enclosed by $T$ is $3\sqrt{3}r^2$. Whence, the derivative with 
respect to $r$ of the area enclosed by $T$ is $6\sqrt{3}r$, which is the length of $T$ as a curve. This method works for any regular $n$-gon. We pose 
this as an exercise. The reader may consult \cite{JM} for this result.
\bigskip

\begin{exercise}\rm{Let $W$ be a regular $n$-gon and let $r$ be the radius of the largest circle that can be inscribed in $W$. Let $A(n, r)$ be the 
area of the region enclosed by $W$ and let $L(n, r)$ be the length of $W$ as a curve. Show that the derivative of $A(n, r)$ with respect to $r$ is $L(n, r)$.}  
\end{exercise}
\bigskip

\noindent These ideas extend to three dimensions in the cases of a cube and other solids such as a tetrahedron, whose four faces are identical 
equilateral triangles. In the three dimensional case, the radius of the largest inscribed sphere is the variable $r$. For the cube,
the sidelength is $2r$, the volume enclosed by the cube is $8r^3$ and the surface area is $24r^2$, which is the derivative with respect to $r$ of the 
volume formula here. The details for our tetrahedron are more cumbersome, but the idea is the same. With these regular curves in $\mathbb{R}^2$ 
and regular surfaces in $\mathbb{R}^3$ that are topologically equivalent to a sphere, formulae for curve length, area enclosed, surface area and volume 
enclosed are not difficult to calculate. But, we really do not need regularity to makes things work. We first set our terminology. For an $n$-gon, the \textit{edges} 
are the line segments that join the vertices and for a surface that is made up of planar regions, the planar regions are called \textit{facets}. So, a triangle has 
three edges and a tetrahedron has four facets. The essential requirement that we need to make things work is that there is a largest inscribed circle (or a largest 
inscribed sphere, in the three dimensioinal case) and that this circle (or sphere) be tangent to each edge (or each facet). The next two lemmas are straightforward, 
but they get to the central issue here. We omit their proofs since they are well-known. For a vector $\vec{x} = (x_1, x_2, ..., x_n)$ in
$\mathbb{R}^n$, the Euclidean norm of $\vec{x}$ is given by $\|\vec{x}\| := \sqrt{x_1^2 + x_2^2 + \cdot\cdot\cdot x_n^2}$.
\bigskip

\begin{lemma} Let $\mathcal{L}$ be a line in the Cartesian plane $\mathbb{R}^2$ that does not pass through the origin and let $c$ be a positive real number, 
not equal to $1$. Then $c\mathcal{L} := \{c\vec{x}: \vec{x}\in\mathcal{L}\}$ is a line in $\mathbb{R}^2$ that is parallel to $\mathcal{L}$, there exists $\vec{x_0}$ in 
$\mathcal{L}$ such that $\|\vec{x_0}\| \leq |\vec{x}\|$ for all $\vec{x}\in\mathcal{L}$ and the distance between $\mathcal{L}$ and $c\mathcal{L}$ is $|1-c|\|\vec{x_0}\|$.
\end{lemma}
\bigskip

\begin{lemma} Let $\mathcal{P}$ be a plane in $\mathbb{R}^3$ that does not pass through the origin and let $c$ be a positive real number, not equal to $1$.
Then $c\mathcal{P} := \{c\vec{x}: \vec{x}\in\mathcal{P}\}$ is a plane in $\mathbb{R}^3$ that is parallel to $\mathcal{P}$, there exists $\vec{x_0}$ in 
$\mathcal{P}$ such that $\|\vec{x_0}\| \leq \|\vec{x}\|$ for all $\vec{x}\in\mathcal{P}$ and the distance between $\mathcal{P}$ and $c\mathcal{P}$ is $|1-c|\|\vec{x_0}\|$.
\end{lemma}
\bigskip

\noindent With regard to our lemmas, notice that the circle about the origin in $\mathbb{R}^2$ of radius $\|\vec{x}_0\|$ lies tangent to the line $\mathcal{L}$ and 
intersects $\mathcal{L}$ at $\vec{x_0}$. Moreover, any circle about the origin of radius less than $\|\vec{x}_0\|$ has no intersection with $\mathcal{L}$. And, similarly, 
the sphere about the origin in $\mathbb{R}^3$ of radius $\|\vec{x}_0\|$ lies tangent to the plane $\mathcal{P}$ and intersects $\mathcal{P}$ at  
$\vec{x_0}$. Moreover, any sphere about the origin of radius less than $\|\vec{x}_0\|$ has no intersection with $\mathcal{P}$. Let $W$ be an $n$-gon (not necessarily regular)
in $\mathbb{R}^2$ and suppose that there is a largest circle that can be inscribed inside $W$ and that this circle lies tangent to each of the edges of $W$. Without loss of generality, 
we may assume that this circle is centered at the origin. For any positive real number $c$, let $cW = \{c\vec{x}: \vec{x}\in W\}$. Then $cW$ is an $n$-gon that has edges parallel 
to corresponding edges of $W$ and the distances between these corresponding parallel edges are all the same, independent of the edge; cf., Lemma 1.2. Continuing with this $W$, 
the length of $W$ as a curve is the product of the length of $\frac{1}{r}W$ with $r$ and the area enclosed by $W$ is the product of the area enclosed by $\frac{1}{r}W$ with $r^2$; 
where $\frac{1}{r}W$ is the multiple of $W$ whose largest inscribed circle has radius equal to $1$. So, with $W$ fixed, the length of the curve $W$ and the area of the region 
enclosed by $W$ are functions of the variable $r$, alone. Similar observations carry over to surfaces in $\mathbb{R}^3$ that are topologically equivalent to a sphere and that are 
made up of finitely many facets, each of which has piecewise smooth boundary.
\bigskip

\begin{theorem} Let $W$ be an $n$-gon in $\mathbb{R}^2$ and suppose that there is a largest circle that can be inscribed inside $W$ and that this circle lies tangent to each 
edge of $W$. Let $r$ be the radius of that circle, let $A(r)$ be the area of the region enclosed by $W$ and let $L(r)$ be the length of $W$ as a curve. Then the derivative of $A(r)$ 
with respect to $r$ is $L(r)$. Similarly, let $\Omega$ be a surface in $\mathbb{R}^3$ that is topologically equivalent to a sphere. Further, suppose that $\Omega$ has finitely many 
planar facets each having piecewise smooth boundary and suppose that there is a largest sphere that can be inscribed inside $\Omega$ and that this sphere lies tangent to each facet 
of $\Omega$. Let $r$ be the radius of that sphere, let $V(r)$ be the volume of the region bounded by $\Omega$ and let $\sigma(r)$ be the surface area of $\Omega$. Then the 
derivative of $V(r)$ with respect to $r$ is $\sigma(r)$.
\end{theorem}   

\begin{proof}
We establish the result in the cases that $W$ is a triangle and that $\Omega$ is a tetrahedron. These will be quite illustrative of how things go in general. 

Let $W$ be any triangle. Place a small circle inside $W$, near a corner of $W$, so that it is tangent to the two edges of $W$
which form that corner. Move the circle away from the corner and expand the circle so that it remains tangent to these two edges until the circle 
meets the third edge of $W$. We thus find a circle inscribed in $W$ that is tangent to all three edges of $W$; which, by construction, is the largest inscribed circle
in $W$. This is sometimes called the inscribed circle, or \textit{incircle}, of the triangle $W$. Let $r$ be the radius of this circle. As before, we may assume that this circle is centered 
at the origin in $\mathbb{R}^2$. For $0 < s < r$, consider $\frac{s}{r}W$, which is a triangle inside $W$ that has sides parallel to corresponding sides of $W$ and has its 
largest inscribed circle equal to $\{\vec{x}: \|\vec{x}\| = s\}$. Moreover, the gaps between the sides of $W$ and the corresponding sides of $\frac{s}{r}W$ are all 
equal to $r-s$; see Figure. 

\setlength{\unitlength}{1.5cm}
\begin{picture}(6, 4)
 \linethickness{0.5mm}
  \put(2, 3){\line(1, -2){.7}}
  \put(2.69, 1.6){\line(3, 1){1.76}}
  \put(2, 3){\line(3, -1){2.45}}
  \put(2.9, 2.75){$\swarrow$}
  \put(3.2, 3.2){$\frac{s}{r}W$}
  \put(4.9, 2.3){$\swarrow$}
  \put(5.3, 2.7){$W$}
  \put(1.7, 3.25){\line(1, -2){.915}}
  \put(2.62, 1.42){\line(3, 1){2.27}}
  \put(1.7, 3.25){\line(3, -1){3.19}}
  \put(2.5, 0){\mbox{\Large{Figure}}}
  \end{picture}
  \vspace*{.5in}

\noindent Let $A(r)$ be the area of the region enclosed by $W$ and 
let $A(s)$ be the area of the region enclosed by $\frac{s}{r}W$. Let $L(r)$ be the length of the curve $W$ and let $L(s)$ be the length of the curve $\frac{s}{r}W$.
Notice that
$$L(s)[r-s] < A(r) - A(s) < L(r)[r-s]\,\,\, \mbox{and that}\,\,\, L(s)\rightarrow L(r),\,\,\,\mbox{as}\,\,\, s\rightarrow r.$$
Therefore, 
$$(A(r) - A(s))/(r-s)\rightarrow L(r),\,\,\,\mbox{as}\,\,\, s\rightarrow r.$$
The same type of argument holds for $s > r$, and we find that the derivative of $A(r)$ with respect to $r$ is $L(r)$.

We now let $\Omega$ be any tetrahedron. Our proof here mimicks the proof for the triangle. Place a small sphere inside $\Omega$, near a corner of $\Omega$, so that it is 
tangent to the three facets of $\Omega$ which form that corner. Move the sphere away from the corner and expand the sphere so that it remains tangent to these three 
facets until it meets the fourth facet of $\Omega$. In this way we find an inscribed sphere inside $\Omega$ that is tangent to all four facets of $\Omega$. By construction, 
this sphere is the largest inscribed sphere in $\Omega$. We may assume that this sphere is centered at the origin in $\mathbb{R}^3$ and we let $r$ be the radius of this
sphere. Proceeding as in the case of the triangle, but with $\Omega$ in place of $W$, let $V(r)$ be the volume of the region enclosed by $\Omega$ and let $V(s)$ be the
volume of the region enclosed by $\frac{s}{r}\Omega$. Let $\sigma(r)$ be the surface area of $\Omega$ and let $\sigma(s)$ be the surface area of $\frac{s}{r}\Omega$. 
Then we have:
$$\sigma(s)[r-s] < V(r) - V(s) < \sigma(r)[r-s]\,\,\, \mbox{and that}\,\,\, \sigma(s)\rightarrow \sigma(r),\,\,\,\mbox{as}\,\,\, s\rightarrow r.$$
Therefore, 
$$(V(r) - V(s))/(r-s)\rightarrow \sigma(r),\,\,\,\mbox{as}\,\,\, s\rightarrow r.$$
The same type of argument holds for $s > r$, and we find that the derivative of $V(r)$ with respect to $r$ is $\sigma(r)$.

\end{proof}

The proof of Theorem 1.4 outlines the argument in general and establishes the result for any triangle (in the two dimensional case)
and any tetrahedron (in the three dimensional case). Unfortunately, the strategy of defining things in terms of a largest inscribed circle 
or sphere does not work well in general. It does not even work for nonsquare rectangles. For an approach that works in greater generality
one may consult \cite{JT} and articles that reference this paper.

There is another idea that generalizes to curves 
like rectangles, indeed to most piecewise smooth Jordan curves (cf., \cite{M} and \cite[Section 4.2]{HG}), though there is no nice parameter $r$, as we have above. 
Take a circle (sphere, in the three dimensional case) of radius $\varepsilon > 0$ and roll it around the curve (or surface) 
so that its center traces out a curve (or surface) on the inside that has gap, on normal lines, to the curve (or surface) in question, 
equal to $\varepsilon$. In this way one can generate a family of curves (or surfaces) which expand out to the curve or surface in question 
and one can take a derivative (as $\varepsilon$ tends to zero) of the area (or volume) of the region bounded by the curve (or surface). 
This certainly works for all rectangles and many curves that are piecewise smooth. We outline this for a 
rectangle with sidelengths $a$ and $b$. For small $\varepsilon > 0$, roll a circle on the inside of the rectangle, up against the curve, 
to trace out a Jordan curve that is inside the rectangle. That Jordan curve, in this case, is actually a rectangle of sidelengths $a - 2\varepsilon$ 
and $b-2\varepsilon$; provided $\varepsilon$ is sufficiently small. Taking the ``derivative'' of the area, 
$$\lim_{\varepsilon\rightarrow 0}[ab- (a-2\varepsilon)(b-2\varepsilon)]/\varepsilon = 2a +2b,$$
we have the perimeter of our original rectangle. So our theme here prevails, though this method is less natural than the earlier one.
\vspace*{.5in}

\noindent\textbf{Acknowledgement.} We are grateful to John Duncan for comments that have helped to improve the exposition
of this paper.

\end{document}